\documentclass[12pt]{amsart}

\usepackage{amsthm}
\usepackage{amsmath}
\usepackage[utf8]{inputenc}
\usepackage[top=72pt,bottom=96pt,left=72pt,right=72pt]{geometry}
\usepackage{color}

\numberwithin{equation}{section}

\def\Aut{\mathrm{Aut }}
\def\C{\mathbb{C}}

\def\End{\mathrm{End}}

\def\Gr{\mathrm{Gr}}

\def\H{\mathbb{H}}
\def\Hol{\mathrm{Hol}}

\def\R{\mathbb{R}}
\def\Ric{\mathrm{Ric}}
\def\ric{\mathrm{Ric}}
\def\volume{\mathrm{Vol}}
\def\scal{\mathrm{scal}}
\def\dv{\text{ }dV}

\def\SO{\mathrm{SO}}

\def\Sp{\mathrm{Sp}}

\def\SU{\mathrm{SU}}

\def\tr{\mathrm{tr}}
\def\U{\mathrm{U}}

\def\T{\mathrm{T}}
\def\Sym{\mathrm{Sym}}
\def\<#1,#2>{\langle\,#1,\,#2\,\rangle}
\def\bea{\begin{eqnarray*}}
\def\eea{\end{eqnarray*}}

\def\HM{\mathrm{H}}
\def\EM{\mathrm{E}}
\def\Hol{\mathrm{Hol}}
\def\hol{\mathfrak{hol}}
\def\pr{\mathrm{pr}}
\newtheorem{Lemma}{Lemma}[section]
\newtheorem{Theorem}[Lemma]{Theorem}
\newtheorem{Corollary}[Lemma]{Corollary}
\newtheorem{Proposition}[Lemma]{Proposition}
\theoremstyle{definition}
\newtheorem{Remark}[Lemma]{Remark}

\newtheorem{Definition}[Lemma]{Definition}
\newtheorem{Question}[Lemma]{Question}
\title{On stability and scalar curvature rigidity of quaternion-K\"ahler manifolds}
\author{Klaus Kr\"oncke and  Uwe Semmelmann}
\address{Klaus Kr\"oncke\\
Institutionen för Matematik\\
KTH Stockholm\\
Lindstedtsvägen 25\\
10044 Stockholm, Sweden
 }
\email{kroncke@kth.se}
\address{Uwe Semmelmann\\
 Institut für Geometrie und Topologie\\
 Fachbereich Mathematik\\
 Universität Stuttgart\\
 Pfaffenwaldring 57\\
 70569 Stuttgart, Germany}
\email{uwe.semmelmann@mathematik.uni-stuttgart.de}

\date{\today}
\begin{document}
\begin{abstract}
We show that every quaternion-K\"{a}hler manifold of negative scalar curvature is stable as an Einstein manifold and therefore scalar curvature rigid. In particular, this implies that every irreducible nonpositive Einstein manifold of special holonomy is stable.
In contrast, we demonstrate that there exist quaternion-K\"{a}hler manifolds of positive scalar curvature which are not scalar curvature rigid even though they are semi-stable. 
  \vskip10pt
 \noindent 2020 {\it Mathematics Subject Classification}:
% Primary:
  {53C25, 53C26, 58J50.}\\
 \smallskip
 \noindent
 {\it Keywords}: {quaternion Kähler manifolds, stability}
\end{abstract}
\maketitle
\section{Introduction}
\subsection{Stability of Einstein manifolds}\label{subsec:stability}
Let $g$ be an Einstein metric, that is $\Ric_g=\lambda\cdot g$ for some $\lambda\in\R$.
On a compact manifold, Einstein metrics are critical points of the Einstein-Hilbert functional
$$S(g)=
%\mathrm{vol}(g)^{\frac{2-n}{2}}
\int_M\scal_g\dv_g.$$
on the space $\mathcal{M}_1$ of Riemannian metrics of volume $1$.
It is well-known that Einstein metrics are always saddle points of the functional, and that index and coindex of $D^2_gS$ are both infinite (see e.g.\ \cite{besse}). However if $h$ is a tt-tensor (that is, a trace-free and divergence-free symmetric $2$-tensor), the second variation is given by 
$$D^2_gS(h,h)=-\frac{1}{2}\int_M\langle \nabla^*\nabla h-2\mathring{R}h,h\rangle\dv,\qquad  ( \mathring{R}h)(X,Y) = \sum_i h(R_{X,e_i}Y, e_i) $$
Thus, $D^2_gS$ has finite coindex on the subspace of tt-tensors which we denote by $TT$.
We call $\Delta_E=\nabla^*\nabla-2\mathring{R}$ the Einstein operator.
\begin{Definition}
A compact Einstein manifold is called \emph{stable} if all eigenvalues of $\Delta_E$ on $TT$ are positive and \emph{semi-stable} if all eigenvalues of $\Delta_E$ on $TT$ are nonnegative. 
\end{Definition}
Let $g_t$ be a curve of Einstein metrics through $g_0=g$ and suppose that $h=\dot{g}_0$ is orthogonal to rescalings and the orbit of the diffeomorphism group acting on $g$. Then $h$ is a tt-tensor and $\Delta_Eh=0$. This motivates the following definition:
\begin{Definition}
An element $h\in \epsilon(g):=\ker(\Delta_E|_{TT})$ is called \emph{infinitesimal Einstein deformation}. We call it \emph{integrable}, if it is tangent to a curve of Einstein metrics. We call an Einstein metric integrable, if all of its infinitesimal deformations are integrable.
\end{Definition}
An important problem is to decide whether $g$ is \emph{non-deformable}, i.e. whether the equivalence class $[g]$ under homothetic scaling and pullback by
diffeomorphisms is isolated in the moduli space of Einstein structures.
If $\epsilon(g)$ is trivial, $g$ is obviously non-deformable. If $\epsilon(g)$ is nontrivial, $g$ can still be non-deformable. For example, the product metric on $S^2\times\mathbb{C}P^{2n}$ admits infinitesimal deformations which are all non-integrable \cite{Koi82}. In contrast, all known Ricci-flat metrics on compact manifolds are integrable, see Subsection \ref{subsubsecion:hol_Ricci_flat} below 

Define $E_g:=\ric_g-\frac{1}{n}(\int_M\scal_g\dv_g) g$ and let $D^k_gE(h,\ldots,h)=\frac{d^k}{dt^k}|_{t=0}E_{g+th}$ be its k'th Frechet derivative at $g$.
Suppose that $h\in\epsilon(g)$ is integrable and let $g_t$ be a curve of Einstein metrics such that $g_0=g$ and ${g}'_0=h$. We may assume that $g_t\in\mathcal{M}_1$, so that $E_{g_t}=0$. By differentiating, we obtain
$$ \frac{d}{dt}|_{t=0}E_{g_t}= D_gE(h)=0,\qquad \frac{d^2}{dt^2}|_{t=0}E_{g_t}=D_gE(k)+D^2_gE(h,h)=0,$$
where $k=g_0''$.
The first condition holds because $h\in\epsilon(g)$. The second equation is not automatic and motivates the following definition:
\begin{Definition}
An infinitesimal Einstein deformation $h\in\epsilon(g)$ is \emph{integrable of second order} if and only if there is a symmetric $2$-tensor $k$ field satisfying 
$$D_gE(k)+D^2_gE(h,h)=0.$$
\end{Definition}
If $h$ is integrable, it is necessarily integrable of second order. Furthermore, $h$ is integrable of second order if and only if $D^2_gE(h,h)\perp \epsilon(g)$ \cite[Lem.\ 4.8]{Koi82}. 
Hence, the vanishing of the so-called Koiso obstruction $\Psi(h) := (D^2_gE(h,h), h)_{L^2}$ is a necessary condition for the integrability of second order for an infinitesimal deformation
$h \in \epsilon(g)$.
Integrability of higher order can be defined in a similar way, see e.g.  \cite[Def. 2.4]{BHMW} or \cite[12.38]{besse} for details.

\subsection{Scalar curvature rigidity}
In a recent work \cite{DK24}, Dahl and the first author discovered a deep relation between stability and scalar curvature rigidity. The latter notion is defined in \cite{DK24} as follows:
\begin{Definition}\label{defn:scr}
We call an Einstein metric $\hat{g}$ on a manifold $M$ \emph{scalar curvature rigid} if there is no metric $g$ near $\hat{g}$ such that
\[
g-\hat{g}|_{M\setminus K}\equiv0,\qquad \volume(K,g)=\volume(K,\hat{g})
\]
for some compact set $K\subset M$ which additonally satisfies
\[
\scal_g\geq\scal_{\hat{g}}, \quad 
\scal_g(p)>\scal_{\hat{g}}(p) \text{ for some }p\in M.
\]
\end{Definition}
Note that in the definition, we only consider metrics \emph{near} an Einstein metric. Many rigidity phenomena involving scalar curvature are global on the space of metrics and do not require a volume constraint. For example, as a consequence of the positive mass theorem, see \cite{Wit81}, there is no metric $g$ on $\R^n$ such that $\scal_g\geq0$ and $\scal_g(p)>0$ for some $p\in\R^n$ which agrees with the Euclidan metric outside a compact set. An analogous statement, holds for hyperbolic space, see \cite{CH03}.

A consequence of Theorem 1.4 and Remark 1.6 in \cite{DK24} is the following result:
\begin{Theorem}[\cite{DK24}]\label{thm:cpt_scr}
Let $(M,g)$ be a compact Einstein manifold. If it  is semi-stable and integrable, it  is also scalar curvature rigid. If $(M,g)$ is unstable, it is not scalar curvature rigid.
\end{Theorem}
\begin{Remark}\label{rem:stability_scr}
Because the statement is not explicitly proven in this form in \cite{DK24}, we briefly sketch here why it holds: 
For a parameter $\alpha>0$, let $\lambda_{\alpha}$ be the smallest eigenvalue of the operator $4\alpha\Delta+\scal$. Because the smallest eigenvalue is simple, the functionals $g\mapsto \lambda_{\alpha}(g)$ depend smoothly on the metric.
Einstein metrics are critical points of the functionals $\lambda_{\alpha}$ on $\mathcal{M}_1$, see \cite[Prop.\ 5.3]{DK24}.
% We may assume that $(M,g)$ is not isometric to the standard sphere. Therefore, by the , we may always choose
Let $\mu$ be the smallest nonzero eigenvalue of the Laplace-Beltrami operator, if $(M,g)$ is not isometric to the round sphere, and the second smallest nonzero eigenvalue otherwise.
By the Obata-Lichnerowicz eigenvalue estimate, we can choose $\alpha>0$ so small that 
\begin{equation*}
\left(1-\frac{n-2}{n-1}\alpha\right)\mu>\frac{\scal_g}{n-1}.
\end{equation*}
If $(M,g)$ is semi-stable and the parameter $\alpha$ satisfies the above inequality, $D^2_g\lambda_{\alpha}$ is nonpositive on $T_g\mathcal{M}_1$ and its kernel is spanned by the Lie-derivatives of $g$ and the space $\epsilon(g)$, which can be concluded from the formulas in \cite[Sec.\ 5.4]{DK24}. Integrability implies that there is a manifold $\mathcal{E}_1\subset\mathcal{M}_1$ of Einstein metrics which is tangent to $\mathrm{ker}(D^2_g\lambda_{\alpha})$ and along which $\lambda_{\alpha}$ is constant. Orthogonal to its kernel, $D^2_g\lambda_{\alpha}$ has a uniform negative upper bound.
By a Taylor expansion argument along the lines of \cite[Thm.\ 7.1]{DK24} (where the case of manifolds with boundary is done), one proves that $g$ is a local maximum of $\lambda_{\alpha}$ on $\mathcal{M}_1$, and therefore, it is scalar curvature rigid by \cite[Lem.\ 5.2]{DK24}.
On the other hand, if $(M,g)$ is unstable, we consider the set
\[
\mathcal{C}_1=\left\{g\in\mathcal{M}_1\mid\scal_g\equiv const\right\}. 
\]
 In \cite{Koi79}, Koiso has shown that if $(M,g)$ is a unit-volume Einstein manifold which is not isometric to a round sphere, $\mathcal{C}_1$ is a manifold near $g$ with tangent space
\begin{gather*}
T_g\mathcal{C}_1=\left\{\mathcal{L}_Xg\mid X\in \Gamma(TM)\right\}\oplus TT.
\end{gather*}
Therefore, if $h$ is a TT-tensor such that $(\Delta_Eh,h)_{L^2}<0$, we find a curve $g_t\in\mathcal{C}_1$ such that
\begin{gather*}
\frac{d}{dt}S(g_t)|_{t=0}=0,\qquad \frac{d^2}{dt^2}S(g_t)|_{t=0}=-\frac{1}{2}(\Delta_Eh,h)_{L^2}>0,
\end{gather*}
so that for small $t\neq 0$, $\scal_{g_t}=S(g_t)>S(g)=\scal_g$. Because all $g_t$ have the same volume, this implies that $g$ is not scalar curvature rigid.
\end{Remark}
\begin{Remark}
Stability is a definition of \emph{infinitesimal} nature and often referred to as linear stability in the literature. Theorem \ref{thm:cpt_scr} demonstrates that scalar curvature rigidity is the corresponding \emph{local} definition. Furthermore, under some additional assumptions, scalar curvature rigidity is equivalent to dynamical stability under Ricci flow, see \cite{DK24} and references therein.
\end{Remark}

%\begin{Remark}\textcolor{red}{Rethink and adjust this remark!}
%Theorem \ref{thm:cpt_scr} holds for the following reason. Let
%\[
%\mathcal{C}_1=\left\{g\in\mathcal{M}_1\mid\scal\equiv const\right\}. 
%\]

%Therefore, semi-stability asserts precisely that $D^2_gS$ is negative semi-definite on $T_g\mathcal{C}_1$. Semi-stability and integrability imply that $g$ being a local maximum of $S$ restricted to $\mathcal{C}_1$.
%Finally, \cite[Theorem 1.4]{DK24} states that the latter is equivalent scalar curvature rigidity. 
%\end{Remark}
%\begin{Remark}
%\textcolor{red}{These considerations suggest that scalar curvature rigidity can be regarded as a local version of stability. Expand this remark suitably. Perhaps refer to Ricci flow as well}
%\end{Remark}

Note that Definition \ref{defn:scr} does assume neither compactness nor completeness. To formulate an analouge of Theorem \ref{thm:cpt_scr} for open manifolds, we extend the definition of stability by saying that a (possibly open) Einstein manifold is stable if there exists a constant $C>0$ such that $\Delta_Eh|_{TT}\geq C$ holds in the $L^2$-sense, that is, $(\Delta_Eh,h)_{L^2}\geq C\left\|h\right\|_{L^2}^2$ for all tt-tensors $h$ with compact support. It is semi-stable, if $(\Delta_Eh,h)_{L^2}\geq0$ for all tt-tensors $h$ with compact support. 
The analogue of Theorem \ref{thm:cpt_scr} for open manifolds now reads as follows:
\begin{Theorem}[\cite{DK24}]\label{thm:open_scr}
Let $(M,g)$ be an open Einstein manifold of nonpositive scalar curvature which is semi-stable. Then it is scalar curvature rigid.
\end{Theorem}
This result is a consequence of Theorem 1.9 and Remark 1.10 in \cite{DK24}. The full form of \cite[Thm. 1.9]{DK24} states an equivalence of linear stability and scalar curvature rigidity without any restriction to the scalar curvature, provided that some additional conditions do hold. We formulated Theorem \ref{thm:open_scr} in this form in order to keep the statement simple.
Note that in constrast to the compact case, we do not need to assume the integrability condition. The essential reason is that according to Definition \ref{defn:scr}, we only need to consider perturbations of $g$ which are supported on compact subsets $K\subset M$.
Because $M$ is open, we find for each compact $K\subset M$ another open subset $N$ such that $K\subset N\subset M$. By domain monotonicity of the smallest Dirichlet eigenvalue, $N$ is stable and thus not contains any infintesimal deformations.

% Rayleigh quotient
%\[ \inf_{h\in C^{\infty}_c(N)} \frac{(\Delta_Eh,h)_{L^2}}{\left\| h\right\|_{L^2}^2}
%\] 
% is strictly contained in $M$ and $\Delta_E|_{TT}$ will be strictly positive on $K$.

In this paper, we show  that the integrability condition in Theorem \ref{thm:cpt_scr} is essential:
\begin{Theorem}\label{thm:nonint_nonscr}
Let $(M,g)$ be a compact Einstein manifold and suppose that it admits infinitesimal deformations which are not integrable of second order. Then $(M,g)$ is not scalar curvature rigid.
\end{Theorem}
The failure of integrability of second order allows us to construct a curve $g_t$ of metrics of volume $1$ and constant scalar curvature along which we can increase $S(g_t)$, which is by the assumptions on $g_t$ equal to the constant function $\scal_{g_t}$.
This will be shown in Section \ref{sec:non-integrable}.
\subsection{Special holonomy metrics}\label{subsec:special_holonomy}
An oriented Riemannian manifold $(M,g)$ is said to be of special holonomy if its holonomy group is a proper subgroup of $SO(n)$. If $(M,g)$ is simply-connected, irreducible and nonsymmetric, there are six possible types of holonomy groups. For five of them (with the K\"{a}hler case as the only exception), $(M,g)$ is nessecarily Einstein.
Conversely, many Einstein metrics are of special holonomy.  It is thus natural to consider the questions of stability and scalar curvature rigidity specifically for special holonomy metrics. For all holonomy types except one, answers are already given in the literature and we give a brief overview over these results in the following. 
In this article, we study stability and scalar curvature rigidity of the remaining holonomy type, which are the quaternion-K\"{a}hler manifolds.

\subsubsection{Ricci flat metrics}\label{subsubsecion:hol_Ricci_flat}
In the Ricci-flat case, every special holonomy metric has a universal cover which admits a parallel spinor. Thus, they are semi-stable by \cite{DWW05}. This follows essentially from \cite{Wang91}, although it is not written there explicitly.
 The integrability for compact manifolds of special holonomy is also a well-known result, see \cite{Wang91}. Therefore, all these metrics are scalar curvature rigid by Theorem \ref{thm:cpt_scr}. All known examples of compact Ricci-flat manifolds are of special holonomy and it is a long-standing open question whether other examples exist.

\subsubsection{K\"{a}hler-Einstein metrics of nonzero scalar curvature}
Using $spin^c$-geometry, Dai, Wang and Wei showed that K\"{a}hler-Einstein metrics of negative scalar curvature are semi-stable \cite{DWW07}. This already follows from work by Koiso (\cite{Koi83}, see also the discussion in \cite[p. 361--364]{besse}), but is not stated there explicitly. 

It is still open whether K\"{a}hler-Einstein metrics of negative scalar curvature are integrable. It has however been shown by Nagy and the second author in \cite{NS23} that all infinitesimal Einstein deformations of such metrics are integrable of second order.
In dimension four, this also follows from combining results of LeBrun, Dahl and the first author with Theorem \ref{thm:nonint_nonscr}:
By \cite{LeB99}, compact four-dimensional K\"{a}hler-Einstein manifolds of negative scalar curvature maximise the Yamabe invariant. Therefore by \cite[Theorem 1.1]{DK24}, they are scalar curvature rigid. On the other hand, if there were infitesimal Einstein deformations which are not integrable of second order, Theorem \ref{thm:nonint_nonscr} would imply that the metric is not scalar curvature rigid, which leads to a contradiction.

%\end{Remark}
Because the above mentioned result by Nagy and the second author \cite{NS23} holds in any dimension, this raises the following question.
\begin{Question}
Are all compact  K\"{a}hler-Einstein manifolds with negative Einstein constant scalar curvature rigid?
\end{Question}
On the other hand, K\"{a}hler-Einstein manifolds of positive scalar curvature and $h_{1,1}(M)>1$ are unstable and therefore not scalar curvature rigid (see \cite{CHI}). The easiest way to construct such an example is by taking a product of two K\"{a}hler-Einstein manifolds. The above mentioned example $S^2\times\mathbb{C}P^{2n}$ is a positive K\"{a}hler-Einstein manifold whose infinitesimal deformations are all non-integrable.

\subsubsection{Symmetric spaces}\label{sym}
Let $M=G/K$ be an irreducible symmetric space with the metric $g$ induced by the Killing form of $G$. The holonomy of the Levi-Civita connection can
be identified with $K$, i.e. we may consider symmetric spaces as manifolds with special holonomy. Then $g$ is Einstein with non-negative sectional
curvature if $M$ is of compact type and with non-positive sectional curvature if $M$ is of non-compact type. In the later case the metric $g$ is stable (see
\cite[12.73]{besse}). If $M$ is of compact type then the stability type was decided by Koiso in   \cite{Koi80}, see also \cite{SW2} for the treatment of a
few remaining cases. It turns out that all symmetric spaces of compact type are semi-stable, with the exception of the following unstable cases:
$\Sp(r), \Sp(n)/\U(n)$ and the Grassmannians $\Gr_2 (\R^5)$ and $\Gr_r(\H^{r+s})$ for $r,s \ge 2$. Infinitesimal Einstein deformations exist on the
following spaces:
$$
\SU(n), \;  \;   \SU(n)/\SO(n),  \;  \;    \SU(2n)/\Sp(n),    \;  \;   \SU(p+q)/\mathrm S(\U(p)\times \U(q)), \;  \;    \mathrm E_6/\mathrm F_4
$$
where $n\ge 3$ and $p,q \ge 2$. For the four $\SU(m)$ quotients in this list the space of infinitesimal Einstein deformations integrable of second order
can be explicitly described and turns out to be a proper subspace of $\epsilon(g)$ (see \cite{BHMW} for $\SU(n)$, \cite{HSS1} for $ \SU(n)/\SO(n)$ and $  \SU(2n)/\Sp(n) $
and \cite{HSS2, NS23} for $\SU(p+q)/\mathrm S(\U(p)\times \U(q))$). It follows from  Theorem \ref{thm:nonint_nonscr} that all these spaces are not scalar curvature rigid.
For the space  $ \mathrm E_6/\mathrm F_4 $ it is known (see \cite{Koi82}) that all infinitesimal Einstein deformations are  integrable of second order. Hence, this
is the only symmetric space  for which it is still not known whether it is scalar curvature rigid or not.

\medskip

Note that reducible symmetric spaces are unstable, as it is always the case for products of Einstein manifolds. 

\subsubsection{Quaternion-K\"{a}hler manifolds}
Recall that quaternion-K\"ahler manifolds are defined
by the condition that the holonomy is a subgroup of $\Sp(1)\cdot\Sp(n) \subset \SO(4n)$. 
For $n \ge 2$, they are  automatically Einstein. Furthermore, they are
deRham irreducible if the scalar curvature is different from zero. If the scalar curvature
vanishes the holonomy reduces further to $\Sp(n)$, i.e.\ the manifold is then hyper-K\"ahler (see \cite[Chap.\ 14]{besse} for further details).
Our main result in the case of negative scalar curvature is as follows.

\begin{Theorem}\label{qk-stable}
Every quaternion-K\"{a}hler manifold $(M,g)$ of negative scalar curvature is stable. In particular, it is scalar curvature rigid and, if $M$ is compact, the metric $g$ is non-deformable as an Einstein metric.
\end{Theorem}
Using representations of the Holonomy group, we construct an parallel bundle embedding $\Phi:\mathrm{Sym}^2T^*M\to \Lambda^4 T^*M$ via which $\Delta_E+2\frac{\scal_g}{\dim M}$ corresponds to the (nonnegative) Hodge-Laplace operator. This construction will be explained in Section \ref{sec_stability}.

%\begin{Example}
%\textcolor{red}{Elaborate a bit on examples}
%\end{Example}
\medskip

The first examples of quaternion-K\"ahler manifolds of negative scalar are the noncompact Wolf spaces. These are symmetric spaces and there is one such
space for each simple  Lie algebra, e.g. the quaternionic hyperbolic  space $\H H^n$. There are many more examples with negative 
scalar curvature, homogeneous and non-homogeneous, see e.g. \cite{C1, C2} and references therein. However, the only known compact examples are compact quotients of 
the non-compact Wolf spaces. 

\medskip

In the case of positive scalar curvature the compact Wolf spaces are the only known examples and it is an open question whether there are other examples. 
All these symmetric examples are stable, with the exception of the complex $2$-plane  Grassmannian $\mathrm{Gr}_2(\C^{n+2})$  which is semi-stable, i.e. it admits 
infinitesimal Einstein deformations (see \cite{SW2}) and, as already mentioned in  Subsection \ref{sym}, some of them are not integrable of second order. Thus we have

%In the case of odd $n$, Nagy and the second author showed  in  \cite{NS23} that all infinitesimal Einstein 
%deformations are not integrable of second order. \textcolor{red}{mention case $n$ even?} In combination with Theorem \ref{thm:nonint_nonscr}, we obtain
\begin{Corollary}\label{thm:qk_psc}
There are quaternion-K\"{a}hler manifolds of positive scalar curvature which are not scalar curvature rigid.
\end{Corollary}
We see a pattern for quaternion-K\"{a}hler manifolds which is similar to other special holonomy Einstein manifolds and with our results, we can summarize the discussion in Subsection \ref{subsec:special_holonomy} as follows: All Einstein metrics of special holonomy and negative scalar curvature are semi-stable. Among Einstein manifolds of positive scalar curvature, we find for each type of special holonomy type (K\"{a}hler, quaternion-K\"{a}hler, symmetric spaces) exmples which are not scalar curvature rigid.
%%%%%%%%
%\medskip
%
%This paper is structured as follows: In Section \ref{sec_stability}, we discuss the stability of quaternion-K\"{a}hler manifolds of negative scalar curvature and prove Theorem \ref{qk-stable}. In Section \ref{sec:non-integrable}, we prove Theorem \ref{thm:nonint_nonscr} which directly implies Corollary \ref{thm:qk_psc}.

%
%\section{Preliminaries}
%
\section{Stability of quaternionic K\"{a}hler manifolds}\label{sec_stability}

\subsection{The standard Laplace operator}
Let $(M^n, g)$ be an oriented Riemannian manifold and let $\Hol:= \mathrm{Hol}(g) \subset \SO(n)$
be the holonomy group of the Levi-Civita connection, with holonomy reduction
$P_\Hol\subset P_{\SO(n)}$, where $P_{\SO(n)}$ is the frame bundle of $(M^n, g)$. 
Then any geometric vector bundle $EM$ on $M$ is associated
to $P_\Hol$ via some representation $\rho: \Hol \rightarrow \Aut (E)$, e.g.
the tangent bundle $TM$ is associated to the standard representation on $T=\R^n$.

Let $\hol(M) \subset \Lambda^2 T M$ be the holonomy algebra bundle defined as associated vector 
bundle via the adjoint representation of $\Hol$ on its Lie algebra $\hol$. 
The differential of the representation $\rho$
 defines fibrewise a parallel bundle map $* : \hol(M) \otimes EM \rightarrow EM$.
We denote with $\nabla$ the connection on sections of $EM$ induced by the Levi--Civita
 connection of $g$. Its  curvature is defined as $R^\nabla_{X, Y} = \nabla^2_{X, Y}
 - \nabla^2_{Y, X}$ for any tangent vectors $X, Y$. It is well-known
 that $R^\nabla_{X, Y} \in \hol (M)$. In particular, the curvature  can be written as $R^\nabla_{X, Y} = R_{X,Y} \circ *$,
 where $R$ is the curvature of the Levi-Civita connection of $g$.
  
The parallel orthogonal projection map $\pr_\hol: \Lambda^2 TM \rightarrow  \hol(M)  \subset \Lambda^2 TM$
allows to define the standard curvature endomorphism for every geometric vector bundle:
$$
q(R) := \frac12 \sum_{i,j} \pr_\hol (e_i \wedge e_j) * R_{e_i, e_j} *  \; \in \End (EM) \ .
$$
The definition of $q(R)$ is independent of the choice of a local orthonormal frame $\{e_i\}$. The
main example comes from the classical Weitzenb\"ock formula for the Hodge-Laplace operator
$$
\Delta = d d^* + d^* d = \nabla^*\nabla + q(R) \ . 
$$
In particular, is $q(R)$  the Ricci tensor on $1$-forms. On symmetric $2$-tensors we have
$$
q(R) = 2 \mathring{R} + 2 \Ric \ ,
$$
where $\Ric$ acts as a derivation, e.g. as $2\lambda$ on an Einstein manifold with Einstein constant $\lambda$.
%and  $ \mathring{R}$ is defined by $ ( \mathring{R}h)(X,Y) = \sum_i h(R_{X,e_i}Y, e_i)$.

The standard Laplace operator $\Delta = \Delta_\rho$ acting on sections of a geometric
vector bundle $EM$ is defined as the sum $\Delta_\rho = \nabla^* \nabla + q(R)$. Then
$\Delta_\rho$ coincides with the Hodge-Laplace operator $\Delta$  if $\rho$ is the 
restriction to $\Hol$ of the standard representation of $\SO(n)$  on forms and similarly with the 
Lichnerowicz Laplacian $\Delta_L$ on symmetric tensors  if $\rho$ is the restriction of
the standard representation of $\SO(n)$ on symmetric tensors. The notation standard Laplace operator
was introduced in \cite{SW1}. A similar construction can be found in \cite[Ch. 1.I]{besse}

The most important property of the standard Laplacian $\Delta_\rho$ is that it commutes with
parallel bundle maps, i.e. with maps induced by $\Hol$-equivariant maps between $\Hol$-representations.
In particular, if $EM \subset \Lambda^k T^*M$ is a parallel subbundle, then the restriction  of the Hodge-Laplace operator
to sections of  $EM$ coincides with the standard Laplace operator
$\Delta_\rho$ of the bundle $EM$. The same is true for any parallel subbundle of the
bundle of symmetric tensors and the Lichnerowicz Laplacian. As a consequence we have
$\Delta_L \ge 0$ on all parallel subbundles $EM \subset \Sym^* TM$ which are also parallel
subbundles of the form bundle $\Lambda^* T^*M$. Note, that $EM$ is a parallel subbundle
of $\Lambda^k T^*M$ if and only if $E \subset \Lambda^k T$ is an $\Hol$-invariant subspace 
with respect to the standard representation of $\SO(n)$ on $\Lambda^kT$ restricted to $\Hol$.

\subsection{Quaternion-K\"ahler manifolds of negative scalar curvature}

Let $(M^{4n}, g)$ be a quaternion-K\"ahler manifold. Then  $\Hol(g) \subset \Sp(1)\cdot\Sp(n) \subset \SO(4n)$
and equivalently there
are locally defined almost complex structures $I, J, K$ compatible with the
metric, satisfying the quaternionic relation $I J = K$ and spanning a globally
defined parallel subbundle of $\End (\T M)$. Such a triple $\{I, J, K\}$ is
called a local quaternionic frame.

\medskip

The standard 
representations of $\Sp(1)$ on $\HM := \mathbb H^1$ and of $\Sp(n)$
on $\EM:= \mathbb H^n$ give rise to locally defined vector bundles again
denoted with $\HM$ and $\EM$.  Any even number of factors in the tensor product leads to
globally defined bundles, e.g. the complexified tangent bundle can be 
written as $\T M^\C = \HM \otimes \EM$. Especially important will be the 
following decomposition.
%%%%%%%%%%%%
\begin{Lemma}
The vector bundle of symmetric $2$-tensors decomposes into the direct sum of three
globally defined parallel subbundles:
\begin{equation} \label{deco}
\Sym^2 \T^* M^\C \;\cong\; (\Sym^2 \HM^* \otimes \Sym^2 \EM^*) \,\oplus\, \Lambda^2_0\EM^*  \;\oplus\; \C \ .
\end{equation}
Here $\Lambda^2_0 \EM$ denotes the space of primitive $2$-forms on $\EM$, i.e. $2$-forms orthogonal
to the symplectic form $\sigma_\EM$. The trivial bundle $\C$ is spanned by the metric $g= \sigma_H \otimes \sigma_E$.
\end{Lemma}
%%%%%%%%%%%%%%
\proof
Decomposition \eqref{deco} corresponds to a decomposition
of the $\Sp(1)\cdot\Sp(n)$ representations $\Sym^2(\HM^* \otimes \EM^*)$ into irreducible summands.
Hence all three summands in the decomposition define parallel subbundles  of the bundle of symmetric 
$2$-tensors. 
%The irreducibility of these representations is well-known.

The decomposition of $\Sym^2(\HM^* \otimes \EM^*)$ is  well-known and follows from general formulas for 
Schur functors (see \cite{FH}). It can also
be proved by first giving explicit embeddings of the three summands and then comparing
the dimensions.  The embedding of the summands can be 
described as follows. Consider  $\alpha_H \in \Sym^2 \HM^*$ and $\alpha_E \in \Sym^2 \EM^*$,
then $\alpha= \alpha_H \otimes \alpha_E$ defined by 
$\alpha(a \otimes e, \tilde a \otimes \tilde e) =\alpha_H(a, \tilde a) \alpha_E(e, \tilde e) $ is
obviously in  $\Sym^2 \T M^\C$. For any  $\eta \in \EM^* \otimes \EM^*$ we define
$\eta_T = \sigma_H \otimes \eta$ by  $\eta_T(a \otimes e, \tilde a \otimes \tilde e) =\sigma_H(a, \tilde a)\eta(e, \tilde e)$.
If $\eta$ is skew-symmetric, i.e. in $\Lambda^2E$, then $\eta_T$ is symmetric. Thus defining
the embedding of the second and third summand. Recall that the Riemannian metric is given
as $g= \sigma_H \otimes \sigma_E =  (\sigma_E)_T$. 
For later use we also note, that $\eta_T$ is skew-symmetric if $\eta $ is symmetric.
\qed

\medskip

The proof of Theorem \ref{qk-stable} is based on the following simple observation.

\begin{Lemma}
The three bundles $\Sym^2 \HM^* \otimes \Sym^2 \EM^*, \, \Lambda^2_0\EM^* $ and the trivial bundle $ \C $ 
of the decomposition \eqref{deco} all appear as parallel subbundles of the  bundle of $4$-forms $\Lambda^4 \T M^*$.
\end{Lemma}
\proof
Again it follows directly from  properties of $\Sp(1)\cdot\Sp(n)$ representations. The decomposition of the space of
$k$ forms on $\HM \otimes \EM$ is given in  \cite{G}. It turns out that the  three irreducible summands of  the  $\Sp(1)\cdot\Sp(n)$ representation 
$ \Sym^2 (\HM^* \otimes \EM^*) $ given in \eqref{deco} also appear in the decomposition of the representation  $\Lambda^4  (\HM^* \otimes \EM^*)$.
Hence, there are corresponding parallel subbundles in $\Lambda^4  \T^* M^\C$ and parallel bundle maps identifying 
the three subbundles in $ \Sym^2 \T^* M^\C$ with the corresponding subbundles in $\Lambda^4 \T^* M^\C$. 

The
embeddings can also be described explicitly. For any   symmetric $2$-tensors $ \alpha \in \Sym^2 \HM^*$ resp.  $ \beta \in \Sym^2 \EM^*$ 
we introduce the notation $\alpha_T$ resp. $\beta_T$ for the $2$-forms on $M$ obtained by taking the tensor product with the
symplectic form $\sigma_E$ resp. $\sigma_H$. Note that we obtain the metric as $g = (\sigma_E)_T = (\sigma_H)_T$. Then the map 
$\alpha \otimes \beta \mapsto \alpha_T \wedge \beta_T$ defines an embedding of $\Sym^2 \HM^* \otimes \Sym^2 \EM^*$ into the space  of
$4$-forms $\Lambda^4  \T^* M^\C$.
In order to  describe the embedding of $\Lambda^2 \EM^*$ we fix a local quaternionic frame $\{I, J, K\}$ with corresponding 
K\"ahler forms $\omega_I, \omega_J, \omega_K$, i.e. $\omega_I(\cdot, \cdot) = g(I \cdot, \cdot)$ and similarly
for $\omega_J$ and $\omega_K$.  For any $\eta \in \Lambda^2 \EM^*$ we define a $4$-form  $\hat \eta $ on $M$  by the following formula
$$
\hat \eta \;:=\; \omega_I \wedge (\eta_T(I \cdot, \cdot )) + \omega_J \wedge (\eta_T(J \cdot, \cdot ))  +\omega_K \wedge (\eta_T(K \cdot, \cdot )) \ .
$$
The map $\eta \mapsto \hat \eta$ then defines the   embedding $\Lambda^2 \EM^* \rightarrow \Lambda^4 \T^* M^\C$. In particular, the symplectic form
$\sigma_E \in \Lambda^2\EM^*$ is mapped to the Kraines form $\Omega := \omega_I \wedge \omega_I + \omega_J \wedge \omega_J  + \omega_K \wedge \omega_K$.
The parallel $4$-form $\Omega$ spans the trivial subbundle in $\Lambda^4 \T^*M^\C$.
\qed

%\medskip

%\subsection{Proof of Theorem \ref{qk-stable}}
\begin{proof}[Proof of Theorem \ref{qk-stable}]
Since the standard Laplace operator $\Delta_L$ commutes with parallel bundle maps we see that there is an isometric embedding $\Phi: \Sym^2 \T^* M\to \Lambda^4 \T^* M$ of bundles such that $\Phi\circ \Delta_L=\Delta_H\circ \Phi$.
%
% the Lichnerowicz 
%Laplacian on the three parallel subbundles of $ \Sym^2 \T^* M$ is mapped to the Hodge-Laplace operator  on the corresponding 
%subbundles of  $\Lambda^4 \T^* M$. 
Since $\Delta_E=\Delta_L-2\frac{\scal}{4n}$ it follows that $\Delta_E$ is strictly positive in the $L^2$-sense if the Einstein constant is negative. Indeed, if $h$ is a compactly supported tt-tensor and $\omega=\Phi(h)$, we have
\begin{equation*}
\begin{split}
(\Delta_Eh,h)_{L^2}&=(\Delta_Lh,h)_{L^2}-2\frac{\scal}{4n}\left\|h\right\|_{L^2}^2\\
&=(\Delta_H\omega,\omega)_{L^2}-2\frac{\scal}{4n}\left\|h\right\|_{L^2}^2\geq -2\frac{\scal}{4n}\left\|h\right\|_{L^2}^2,
\end{split}
\end{equation*} 
where we used that the Hodge-Laplace operator is non-negative. 
%Hence, also $\Delta_L$ is a non-negativ operator and in case of negative Einstein constant it follows that the Einstein 
%operator is strictly positive, since $\Delta_E=\Delta_L-2\frac{\scal}{4n}$. 
It follows that quaternion-K\"ahler manifolds of negative scalar curvature
are stable, thus proving Theorem \ref{qk-stable}. 
\end{proof}

 %\medskip
 
 \begin{Remark}
 For quaternion-K\"ahler manifolds of positive scalar curvature the situation is more complicated. Here the stability question is still open. As already mentioned, the
 complex $2$-plane Grassmannian $\mathrm{Gr}_2(\C^{n+2})$ is only semi-stable, as it has a non-trivial space of
  infinitesimal Einstein deformations.
 On a quaternion-K\"ahler manifold of positive scalar curvature one can show that the Lichnerowicz Laplacian on trace-free symmetric $2$-tensors 
 is bounded from below by $2 \, \frac{\scal}{4n}  \, \frac{n+1}{n+2}$ (see \cite{H}).  It is perhaps interesting to note that  $\frac{\scal}{4n}  \, \frac{n+1}{n+2}$ is exactly the Einstein constant of the 
 K\"ahler Einstein metric of the twistor space associated to the quaternion-K\"ahler manifold.
 \end{Remark}

 \section{Non-integrable deformations and scalar curvature rigidity}\label{sec:non-integrable}
Throughout this section, let us assume that $(M,g)$  is an Einstein manifold of volume $1$ and with Einstein constant $\lambda$. Recall that $g$ is a critical point of $S$ on $\mathcal{M}_1$. In order to avoid technical complications, we want to avoid the volume constraint. For this reason, we work instead with the Einstein-Hilbert functional with cosmological constant, given by
\[
S_{\lambda}:g\mapsto S_{\lambda}(g)=\int_M(\scal_g+(2-n)\lambda)\dv_g.
\]
A standard computation shows that
\[
D_gS_{\lambda}(h)=-\left(F_g,h\right)_{L^2},
% =\left(\frac{1}{2}(\scal_g+(2-n)\lambda)g-\Ric_g,h\right)_{L^2},
\]
where
\[
F_g=\Ric_g-\frac{1}{2}\scal_g\cdot g-\frac{1}{2}(2-n)\lambda g.
\]
Note that the Euler-Lagrange equation $F_g=0$ is equivalent to $\Ric_g=\lambda g$.
\subsection{A symmetric trilinear map}
Recall the definitions $E_g=\Ric_g-\frac{1}{n}\left(\int_M\scal_g\dv_g\right) g$ and $\epsilon(g)=\ker(\Delta_E|_{TT})$ from Subsection \ref{subsec:stability}. For a map $\Phi:V\to W$ between Banach spaces, we denote the k'th Frechet derivative at $g$ by $D^k_g\Phi:V\times\ldots \times V\to W$. Our goal in this section is to show that the expression $(D^2_gE(h,k),l)_{L^2}$ defines a totally symmetric trilinear form on $\epsilon(g)$.
\begin{Lemma}\label{lem:linearizations_and_Lambda}
We have
\[
(D^2_gE(h,h),k)_{L^2}=(D^2_gF(h,h),k)_{L^2}
\]
for all $h,k\in \epsilon(g)$.
\end{Lemma}
\begin{proof}
By the first variation of the scalar curvature and the volume element (see e.g.\ \cite[Thm.\ 1.174 and Prop.\ 1.186]{besse}, we have  $D_g\scal(h)=0$ and  $D_g{dV}(h)=0$ for every tt-tensor $h$.
Thus,
\begin{equation*}
\begin{split}
D^2_gE(h,h)&=D^2_g\Ric(h,h)-\frac{1}{n}\left(\int_MD^2_g\scal(h,h)\dv_g+\scal_g\cdot D^2_g{dV}(h,h)\right) g
\\
&\qquad-\frac{2}{n}\left(\int_M D_g\scal(h)D_g{dV}(h)\right)g-\frac{2}{n}\left(\int_M D_g(\scal\dv)(h)\right)h
\\
&=D^2_g\Ric(h,h)-\frac{1}{n}\left(\int_MD^2_g\scal(h,h)\dv_g+\scal_g\cdot D^2_g{dV}(h,h)\right) g
\end{split}
\end{equation*}
and
\begin{equation*}
\begin{split}
D^2_gF(h,h)&=D^2_g\Ric(h,h)-\frac{1}{2}D^2_g\scal\cdot g-D_g\scal(h)\cdot h\\
&=D^2_g\Ric(h,h)-\frac{1}{2}D^2_g\scal(h,h)\cdot g.
\end{split}
\end{equation*}
By taking the scalar product of the two right hand sides with a tensor $k\in\epsilon(g)$ and using that $\langle g,k\rangle=\tr_gh=0$, we see that
\[
\langle D^2_gE(h,h), k \rangle_g = \langle D^2_g\Ric(h,h), k \rangle_g = \langle D^2_gF(h,h), k \rangle_g
\]
and integrating over $M$ yields the desired result.
\end{proof}
This lemma allows us to identify $(D^2_gE(h,k),l)_{L^2}$ as the third variation of $S_{\lambda}$ which is key for the following assertion.
\begin{Proposition}\label{lem:nonint_2nd_order}
The trilinear form
\[
\Phi:\epsilon(g)\times \epsilon(g) \times \epsilon(g) \to\R,\qquad
(h,k,l)\mapsto (D^2E(h,k),l)_{L^2}
\]
is totally symmetric. In particular,
if we find $h,k\in \epsilon(g)$ such that $(D^2E(h,h),k)_{L^2}\neq 0$, we also find $l\in \epsilon(g)$ such that $(D^2E(l,l),l)_{L^2}\neq 0$.
\end{Proposition}
\begin{Remark}
The symmetry of $\Phi$ was already shown in \cite{NS23}. There  the authors obtain as a result of a direct but lengthy calculation an explicit formula 
for the full obstruction $\Phi$ in terms of  the Fr\"olicher-Nijenhuis bracket. It turns out to be symmetric in all three arguments and in particular it recovers the 
Koiso obstruction, i.e. $\Psi(h) = \Phi(h,h,h)$. Hence, Proposition \ref{lem:nonint_2nd_order} gives a simple way to check the explicit formula for $\Phi$ by
first reformulating the Koiso obstruction and then polarising it. In particular, we see that the vanishing of the Koiso obstruction is a necessary and sufficient
condition for integrability of second order.

\end{Remark}

\begin{proof}
We are going to show that the trilinear form of the lemma equals the third variation of $S_{\lambda}(g)$. Let $h,k,l\in \epsilon(g)$ and compute
\begin{equation*}
\begin{split}
\frac{d}{dt}\frac{d}{ds}\frac{d}{dr}S_{\lambda}(g+th+sk+rl)&=
\frac{d}{dt}\frac{d}{ds}D_{g+th+sk+rl}S_{\lambda}(l)\\
&=-\frac{d}{dt}\frac{d}{ds}(F_{g+th+sk+rl},l)_{L^2}\\
&=-\frac{d}{dt}[(D_{g+th+sk+rl}F(k),l)_{L^2}+(F_{g+th+sk+rl},k*l)_{L^2}].
\end{split}
\end{equation*}
Here, the second term on the right hand side comes from differentiating the scalar product and the volume element. We use the $*$-notation to denote a linear combination of tensor products and contractions.
The two terms on the right hand side are computed to be
\begin{equation*}
\begin{split}
\frac{d}{dt}(D_{g+th+sk+rl}F(k),l)_{L^2}&=(D^2_{g+th+sk+rl}F(h,k),l)_{L^2}+(D_{g+th+sk+rl}F(k),h*l)_{L^2},\\
\frac{d}{dt}(F_{g+th+sk+rl},k*l)_{L^2}&=(D_{g+th+sk+rl}F(h),k*l)_{L^2}+(F_{g+th+sk+rl},h*k*l)_{L^2}.
\end{split}
\end{equation*}
Also here, the second terms on the right hand sides comes from differentiating the scalar product and the volume element. 
Now, we evaluate at $t,s,r=0$. Because $g$ is Einstein, $F_g=0$. Because $h,k\in \epsilon(g)$, we have $DF_{g}(k)=DF_g(h)=0$ by the first variation of the Ricci tensor and the scalar curvature, see \cite[Thm.\ 1.174]{besse}. Therefore by the above,
\[
\frac{d}{dt}\frac{d}{ds}\frac{d}{dr}|_{t,s,r=0}S_{\lambda}(g+th+sk+rl)=-(D^2_{g}F(h,k),l)_{L^2}=-(D^2_{g}E(h,k),l)_{L^2},
\]
where the second equation follows from Lemma \ref{lem:linearizations_and_Lambda}.
The left hand side is totally symmetric by Schwarz's theorem, and so is the right hand side. The second statement of the Lemma follows from the identity
\begin{gather*}
6\Phi(h,h,k)=\Phi(k+h,k+h,k+h)+\Phi(k-h,k-h,k-h)-2\Phi(k,k,k),
\end{gather*}
which holds for any totally symmetric triliear form.
\end{proof}
\subsection{Proof of Theorem \ref{thm:nonint_nonscr}}
Before we prove Theorem \ref{thm:nonint_nonscr}, we compute the first three derivatives of $S_\lambda$ along a curve $g_t$ which is tangent to $h\in\epsilon(g)$ at $g_0=g$. In particular, we show that these derivatives depend on $g_0'=h$ but not on higher derivatives of $g_t$.
\begin{Lemma}\label{lem_third_variation}
Let $h\in \epsilon(g)$. Then for every smooth family $g_t$ of metrics with $g_0=g$ and $\frac{d}{dt}|_{t=0}g_t=h$, we have
\begin{gather*}
\frac{d}{dt}|_{t=0}S_{\lambda}(g_t)=0,\qquad \frac{d^2}{dt^2}|_{t=0}S_{\lambda}(g_t)=0,\qquad
\frac{d^3}{dt^3}|_{t=0}S_{\lambda}(g_t)=-(D^2_gE(h,h),h)_{L^2}.
\end{gather*}
\end{Lemma}
\begin{proof}
Let $f(t)=S_{\lambda}(g_t)$. Then we compute
\begin{equation*}
\begin{split}
f'(t)&=D_{g_t}S_{\lambda}(g_t')=-(F_{g_t},g_t')_{L^2},\\
f''(t)&=-(D_{g_t}F(g_t'),g_t')_{L^2}-(F_{g_t},g_t'')_{L^2}+(F_{g_t},g_t'*g_t')_{L^2},\\
f'''(t)&=-(D^2_{g_t}F(g_t',g_t'),g_t')_{L^2}
-2(D_{g_t}F(g_t'),g_t'')_{L^2}
-(D_{g_t}F(g_t''),g_t')_{L^2}
-(F_{g_t},g_t''')_{L^2}\\
&\qquad +(D_{g_t}F(g_t'),g_t'*g_t')_{L^2} +(F_{g_t},g_t'*g_t'')_{L^2}+(F_{g_t},g_t'*g_t'*g_t')_{L^2}.
\end{split}
\end{equation*}
The second term for $f''(t)$ and the
 terms in the last line come from differentiating the scalar product and the volume element. Since $F_{g}=0$, we have $f'(0)=0$. Because $h\in \epsilon(g)$, the first variation of the Ricci tensor and the scalar curvature (see \cite[Thm.\ 1.174]{besse}) yield
 \begin{gather*}
 f''(0)=-(D_{g}F(h),h)_{L^2}=-\frac{1}{2}(\Delta_Eh,h)_{L^2}=0.
 \end{gather*}
 Before we are going to evaluate the third variation, we compute
\begin{equation*}
\begin{split}
 \frac{d}{dt}\frac{d}{ds}S_{\lambda}(g+th+sk)&=\frac{d}{dt}D_gS_{\lambda}(k)\\
 &=-\frac{d}{dt} (F_{g+th+sk},k)_{L^2}
 =- (D_{g+th+sk}F(h),k)_{L^2}+(F_{g+th+sk},h*k)_{L^2}.
\end{split}
\end{equation*}
Evaluating at $t=s=0$ yields
 \begin{gather*}
  \frac{d}{dt}\frac{d}{ds}|_{t,s=0}S_{\lambda}(g+th+sk)=- (D_{g}F(h),k)_{L^2},
 \end{gather*}
 which in particular shows that $DF_{g}$ is symmetric. Thus, $(D_{g}F(g_0''),g_0')_{L^2}=(D_{g}F(g_0'),g_0'')_{L^2}$. In combination with $F_{g}=0$ and $D_{g}F(g_0')=DF_{g}(h)=0$, we obtain
 \begin{gather*}
f'''(0)= -(D^2_{g}F(h,h),h)_{L^2}= -(D^2_{g}E(h,h),h)_{L^2},
 \end{gather*}
 which finishes the proof.
\end{proof}
Now we have all ingredients together to prove Theorem \ref{thm:nonint_nonscr}.
\begin{proof}[Proof of Theorem \ref{thm:nonint_nonscr}]
We may assume that $\mathrm{vol}(M,g)=1$.
If $l\in \epsilon(g)$ is not integrable of second order, we find $k\in \epsilon(g)$ such that 
 \begin{gather*}(D^2_gE(l,l),k)_{L^2}\neq 0.
  \end{gather*}
 By Proposition \ref{lem:nonint_2nd_order}, we also find $h\in \epsilon(g)$ such that
 \begin{gather*}(D^2_gE(h,h),h)_{L^2}\neq 0.
 \end{gather*}
% We now consider the set
%\begin{gather*}
%\mathcal{C}_1=\left\{\tilde{g}\in\mathcal{M}_1\mid \scal_{\tilde{g}}\equiv const\right\}.
%\end{gather*}
%In \cite{Koi79}, Koiso has shown that if $(M,g)$ is not the round sphere, $\mathcal{C}_1$ is a manifold near $g$ whose tangent space is
%\begin{gather*}
%T_g\mathcal{C}_1=\left\{\mathcal{L}_Xg\mid X\in \Gamma(TM)\right\}\oplus TT.
%\end{gather*}
%%see \cite[p. 17]{Kro13}
Because $h\in TT$, it is tangent to the manifold $\mathcal{C}_1$, c.f.\ Remark \ref{rem:stability_scr}.
 Therefore, we find a curve $g_t\in \mathcal{C}_1$ with $g_0=g$ and $g'_0=h$. By
Lemma \ref{lem_third_variation}, we have 
\begin{gather*}
\frac{d}{dt}|_{t=0}S(g_t)=0,\qquad \frac{d^2}{dt^2}|_{t=0}S(g_t)=0,\qquad
\frac{d^3}{dt^3}|_{t=0}S(g_t)=-(D^2_gE(h,h),h)_{L^2}\neq 0.
\end{gather*}
Depending on the sign of the third derivative, we get $S(g_t)>S(g)$ either for $t\in(0,\epsilon)$ 
%(if $\frac{d^3}{dt^3}|_{t=0}S(g_t)>0$) 
or $t\in (-\epsilon,0)$.
% (if $\frac{d^3}{dt^3}|_{t=0}S(g_t)<0$). 
 By definition of $S$ and $\mathcal{C}_1$, we get
\begin{gather*}
\scal_{g_t}>\scal_g,\qquad \mathrm{vol}(M,g_t)=1=\mathrm{vol}(M,g),
\end{gather*}
which finishes the proof of the theorem.
\end{proof}

\section*{Acknowledgments}
The authors are grateful to Gregor Weingart for helpful discussions. The first author would like to thank the G\"{o}ran Gustafsson Foundation for financial support.
 The second author acknowledges the support received by the Special Priority
 Program SPP 2026 {\em Geometry at Infinity} funded by the Deutsche
 Forschungsgemeinschaft DFG. 
 
\bigskip
 
% \newpage
 %

%
\end{document}